\documentclass[reqno]{amsart}
\usepackage{amssymb}
\usepackage{amsfonts}
\usepackage{amsmath}
\usepackage{color}
\usepackage{amsaddr}

\newtheorem{theorem}{Theorem}[section]

\newtheorem{conjecture}[theorem]{Conjecture}

\newtheorem{definition}[theorem]{Definition}

\newtheorem{lemma}[theorem]{Lemma}

\theoremstyle{remark}
\newtheorem{remark}[theorem]{Remark}

\numberwithin{equation}{section}

\begin{document}

\title[Strongly Damped Wave Equations with Dynamic Boundary Conditions]{Global existence of weak solutions for strongly damped wave equations with nonlinear boundary conditions and balanced potentials}

\author[J. L. Shomberg]{Joseph L. Shomberg}



\address{Department of Mathematics and Computer Science, Providence College, Providence, Rhode Island 02918, USA, \\ {\tt{jshomber@providence.edu}
}}

\date{\today}

\begin{abstract}
We demonstrate the global existence of weak solutions to a class of semilinear strongly damped wave equations possessing nonlinear hyperbolic dynamic boundary conditions. 
Our work assumes $(-\Delta_W)^\theta \partial_tu$ with $\theta\in[\frac{1}{2},1)$ and where $\Delta_W$ is the Wentzell-Laplacian.
Hence, the associated linear operator admits a compact resolvent.
A balance condition is assumed to hold between the nonlinearity defined on the interior of the domain and the nonlinearity on the boundary.
This allows for arbitrary (supercritical) polynomial growth on each potential, as well as mixed dissipative/anti-dissipative behavior.
Moreover, the nonlinear function defined on the interior of the domain is assumed to be only $C^0$.
\end{abstract}

\maketitle

\tableofcontents

\section{Introduction}

Our aim in this article is to show the global existence of global weak solutions to the fractional strongly damped wave equation with nonlinear hyperbolic dynamic boundary conditions.
We establish the global existence of weak solutions under a balance condition imposed on the nonlinear terms.
This condition is motivated by \cite[Lemma 3.1]{RBT01}.
In the present article, both nonlinearities are allowed supercritical polynomial growth.
Special attention is given to obtaining the compact resolvent for the associated linear operator which contains (fractional) Wentzell-Laplacians.

Let $\Omega$ be a bounded domain in $\mathbb{R}^3$ with smooth boundary $\Gamma:=\partial\Omega$.
Throughout we assume $\theta\in[\frac{1}{2},1)$, $\omega\in(0,1]$ and $\alpha\in(0,1]$.
We consider the equations in the unknown $u=u(t,x)$,
\begin{align}
\partial^2_tu-\omega\Delta^\theta \partial_tu+\partial_tu-\Delta u+u+f(u)=0 & \quad \text{in} \quad (0,\infty)\times\Omega,  \label{pde1} \\ 
\partial^2_tu+\omega\partial_{\bf n}^\theta \partial_tu+\partial_{\bf n}u-\alpha\omega\Delta_\Gamma \partial_tu+\partial_tu-\Delta_\Gamma u+u+g(u)=0 & \quad \text{on} \quad (0,\infty)\times\Gamma.  \label{pde2}
\end{align}
Additionally, we impose the initial conditions
\begin{align}
u(0,x)=u_0(x) \quad \text{and} \quad \partial_tu(0,x)=u_1(x) \quad \text{at} \quad \{0\}\times\Omega,  \label{ic1}
\end{align}
and
\begin{align}
u_{\mid\Gamma}(0,x)=\gamma_{0}(x) \quad \text{and} \quad \partial_tu_{\mid\Gamma}(0,x)=\gamma_{1}(x) \quad \text{at}\ \{0\}\times\Gamma.  \label{ic2}
\end{align}
Above, $\Delta_\Gamma$ denotes the Laplace-Beltrami operator (cf. e.g. \cite{Ch-La-10}).

We assume $f\in C(\mathbb{R})$ and $g\in C^1(\mathbb{R})$ satisfy the sign conditions
\begin{equation}
\liminf_{|s|\rightarrow\infty}\frac{f(s)}{s}> -M_1, \quad g'(s)\ge -M_2, \quad \forall s\in \mathbb{R},\label{sign}
\end{equation}
for some $M_1,M_2>0$, and the growth assumptions, for all $s\in \mathbb{R}$, 
\begin{equation}
|f(s)|\le \ell_1(1+|s|^{r_1-1}), \quad |g(s)|\le \ell_2(1+|s|^{r_2-1}),  \label{growth}
\end{equation}
for some positive constants $\ell_1$ and $\ell_2$, and where $r_1,r_2\ge 2$. 
In addition, we assume there exists $\varepsilon \in (0,\omega )$ so that
the following balance condition holds,
\begin{equation}
\liminf_{|s|\rightarrow \infty} \frac{f(s)s+\frac{|\Gamma|}{|\Omega|}
g(s)s - \frac{C_\Omega^2|\Gamma|^2}{4\varepsilon|\Omega
|^2}|g'(s)s + g(s)|^2}{|s|^{r_1}} > 0,  \label{balance}
\end{equation}
for $r_1\ge \max \{r_2,2(r_2-1)\}$, where $C_\Omega>0$ is the best Sobolev constant in the following Sobolev-Poincar\'{e} inequality
\begin{equation}
\|u-\langle u\rangle_\Gamma\|_{L^2(\Omega)}\le C_\Omega\|\nabla u\|
_{L^2(\Omega)}, \quad \langle u\rangle_\Gamma := \frac{1}{|\Gamma|}\int\limits_\Gamma tr_D(u) d\sigma,  \label{eq:Poincare}
\end{equation}
for all $u\in H^1(\Omega)$.

Let us provide further context for the balance condition (\ref{balance}) in our setting (also see \cite{RBT01} and \cite{Gal-Shomberg15-2} for other settings).
Suppose that for $|y|\rightarrow\infty,$ both the internal and boundary functions satisfy the following:
\begin{equation*}
\lim_{|y|\rightarrow \infty }\frac{f(y)}{|y|^{r_1-1}}=\left( r_1-1\right) c_f, \quad \lim_{|y|\rightarrow \infty }\frac{g'(y)}{|y|^{r_2-2}}=\left( r_2-1\right) c_g,  
\end{equation*}
for some constants $c_f,c_g\in \mathbb{R}\setminus \{0\}$. 
In particular, there holds
\begin{equation*}
f(y)y\sim c_f|y|^{r_1}, \quad g(y)y\sim c_g|y|^{r_2} \quad \text{as}\ |y|\rightarrow \infty .
\end{equation*}
For the case of bulk dissipation (i.e., $c_f>0$) and anti-dissipative behavior at the boundary $\Gamma $ (i.e., $c_g<0$), assumption (\ref{balance}) is automatically satisfied provided that $r_1>\max \{r_2,2(r_2-1)\}$. 
Furthermore, if $2<r_2<2\left( r_2-1\right)=r_1$ and
\begin{equation*}
c_f>\frac{1}{4\varepsilon}\left( \frac{C_\Omega|\Gamma|c_gr_2}{|\Omega|}\right)^2, 
\end{equation*}
for some $\varepsilon \in (0,\omega )$, then (\ref{balance}) is again satisfied. 
In the case when $f$ and $g$ are sublinear (i.e., $r_1=r_2=2$ in (\ref{growth})), the condition (\ref{balance}) is also automatically satisfied provided that
\begin{equation*}
\left( c_f +\frac{|\Gamma|}{|\Omega|}c_g\right) >\frac{1}{\varepsilon}\left( \frac{C_\Omega |\Gamma| c_g}{|\Omega|}\right)^2 
\end{equation*}
for some $\varepsilon \in \left( 0,\omega \right)$. 

\textbf{Notation and conventions.} Let us introduce some notation and conventions that are used throughout the article. 
Norms in the associated space are clearly denoted $\|\cdot\|_B$ where $B$ is the corresponding Banach space. 
We use the notation $(\cdot,\cdot)_H$ to denote the inner-product on the Hilbert space $H$.
The dual product on $H^*\times H$ is denoted $\langle \cdot,\cdot \rangle_{H^*\times H}$.
The notation $\langle \cdot,\cdot \rangle$ is also used to denote the product on the phase space and various other vectorial function spaces.
Denote by $(u,v)^{tr}$ the vector-valued function $\binom{u}{v}.$ 
In many calculations, functional notation indicating dependence on the variable $t$ is dropped; for example, we will write $u$ in place of $u(t)$. 
Throughout the article, $C>0$ will denote a \emph{generic} constant which may depend on various structural parameters such as $|\Omega|$, $|\Gamma|$, $M_1,$ $M_2$, etc, and these constants may even change from line to line.
Furthermore, $Q:\mathbb{R}_{+}\rightarrow \mathbb{R}_{+}$ will be a generic monotonically increasing function whose specific dependance on other parameters will be made explicit on occurrence.
All of these constants/quantities are {\em independent} of the perturbation parameters $\theta,$ $\alpha$ and $\omega.$

\textbf{Outline of the article.} In the next section we establish the variational formulation of Problem {\textbf{P} and define weak solutions.
A proof of the existence of global weak solutions is developed in Section \ref{s:wp}.
Because of the nature of the balance condition, a continuous dependence type estimate is not available.
The article continues with some remarks on this difficulty and plans for possible further research. 
An appendix contains some explicit characterizations for the fractional Wentzell-Laplacian used throughout the article, as well as a certain compact embedding result that we need to draw upon.

\section{Formulation of the model problem}  \label{s:prelim}

In this section we first recall the Wentzell-Laplacian defined on vectorial Hilbert spaces.
(For this we largely refer to \cite[Section 2]{AreMetPalRom-2003} and \cite[Section 2 and Appendix]{Gal12-1}.)
Following this, we give the basic functional setup in order to formulate the model problem.
We also provide various results pertaining to the problem.

To begin, let $\lambda_\Omega>0$ denote the best constant satisfying the Sobolev inequality in $\Omega$
\begin{equation}  \label{Sobolev}
\lambda_\Omega \int_\Omega u^2 dx \le \int_\Omega (|\nabla u|^2+u^2) dx.
\end{equation}
We will also rely on the Laplace-Beltrami operator $-\Delta_\Gamma$ on the surface $\Gamma.$ 
This operator is positive definite and self-adjoint on $L^2(\Gamma)$ with domain $D(\Delta_\Gamma)$.
The Sobolev spaces $H^s(\Gamma)$, for $s\in\mathbb{R}$, may be defined as $H^s(\Gamma)=D((\Delta_\Gamma)^{s/2})$ when endowed with the norm whose square is given by, for all $u\in H^s(\Gamma)$,
\begin{equation}  \label{LB-norm}
\|u\|^2_{H^s(\Gamma)} := \|u\|^2_{L^2(\Gamma)} + \left\|(-\Delta_\Gamma)^{s/2}u\right\|^2_{L^2(\Gamma)}.
\end{equation}
On the boundary, let $\lambda_\Gamma>0$ denote the best constant satisfying the Sobolev inequality on $\Gamma$
\begin{equation}  \label{bndry-Sobolev}
\lambda_\Gamma \int_\Gamma u^2 d\sigma \le \int_\Gamma\left( |\nabla_\Gamma u|^2+u^2 \right) d\sigma.
\end{equation}

Next, recall that $\Omega$ is a bounded domain of $\mathbb{R}^3$ with boundary $\Gamma$, to which we now assume is of class $\mathcal{C}^2$. 
To this end, consider the space $\mathbb{X}^2 = L^2(\overline{\Omega},d\mu),$ where $d\mu =dx_{\mid\Omega}\oplus d\sigma$ is such that $dx$ denotes the Lebesgue measure on $\Omega$ and $d\sigma$ denotes the natural surface measure on $\Gamma$. 
Then $\mathbb{X}^2=L^2(\Omega,dx)\oplus L^2(\Gamma,d\sigma)$ may be identified by the natural norm
\begin{equation*}
\|u\|_{\mathbb{X}^2}^2=\int_\Omega|u(x)|^2dx + \int_\Gamma
|u(x)|^2d\sigma.
\end{equation*}
Moreover, if we identify every $u\in C(\overline{\Omega})$ with $U=(u_{\mid\Omega},u_{\mid\Gamma})^{tr} \in C(\Omega) \times C(\Gamma)$, we may also define $\mathbb{X}^2$ to be the completion of $C(\overline{\Omega})$ with respect to the norm $\|\cdot\|_{\mathbb{X}^2}$. 
Thus, in general, any function $u\in \mathbb{X}^2$ will be of the form $u=\binom{u_1}{u_2}$ with $u_1\in L^2(\Omega,dx)$ and $u_2\in L^2(\Gamma,d\sigma)$.
It is important to note that there need not be any connection between $u_1$ and $u_2$. 
From now on, the inner product in the Hilbert space $\mathbb{X}^2$
will be denoted by $\langle \cdot ,\cdot \rangle_{\mathbb{X}^2}.$ 
Now we recall that the Dirichlet trace map $tr_D:C^\infty(\overline{\Omega}) \rightarrow C^\infty(\Gamma),$ defined by $tr_D(u) =u_{\mid\Gamma}$ extends to a linear continuous operator $tr_D:H^r(\Omega) \rightarrow H^{r-1/2}(\Gamma),$
for all $r>1/2$, which is onto for $1/2<r<3/2.$ 
This map also possesses a bounded right inverse $tr_D^{-1}:H^{r-1/2}(\Gamma)
\rightarrow H^r(\Omega)$ such that $tr_D(tr_D^{-1}\psi) =\psi,$ for any $\psi \in H^{r-1/2}(\Gamma)$. 
We can thus introduce the subspaces of $H^r(\Omega) \times H^{r-1/2}(\Gamma)$ and $H^r(\Omega) \times H^r(\Gamma)$, respectively, by
\begin{align}
\mathbb{V}_0^r &:= \{U=(u,\gamma) \in H^r(\Omega) \times H^{r-1/2}(\Gamma) : tr_D(u) =\gamma \},  \label{vvv} \\
\mathbb{V}^r &:= \{U=(u,\gamma) \in \mathbb{V}_0^r: tr_D(u) = \gamma \in H^r(\Gamma) \},  \notag
\end{align}
for every $r>1/2,$ and note that $\mathbb{V}_0^r,$ $\mathbb{V}^r$ are
not product spaces. 
However, we do have the following dense and compact embeddings $\mathbb{V}_0^{r_1}\subset \mathbb{V}_0^{r_2},$ for any $r_1>r_2>1/2$ (by definition, this also true for the sequence of spaces $\mathbb{V}^{r_1}\subset \mathbb{V}^{r_2}$). Naturally, the norm on the spaces $\mathbb{V}_0^r,$ $\mathbb{V}^r$ are defined by
\begin{equation}
\|U\|_{\mathbb{V}_0^r}^2:=\|u\|_{H^r(\Omega)}^2+\|\gamma\|_{H^{r-1/2}(\Gamma)}^2, \quad \|U\|_{\mathbb{V}^r}^2:=\|u\|_{H^r(\Omega)}^2+\|\gamma\|_{H^r(\Gamma)}^2.
\label{Vr-norm}
\end{equation}

Here we consider the basic (linear) operator associated with the model problem \eqref{pde1}-\eqref{ic2}, the so-called Wentzell-Laplacian.
Let
\begin{align}
\Delta_W \binom{u_1}{u_2}:=\begin{pmatrix} \Delta u_1-u_1 \\ -\partial_{\bf n}u_1+\Delta_\Gamma u_2-u_2 \end{pmatrix},  \label{A_Wentzell1}
\end{align}
with
\begin{equation}
D(\Delta_W) := \left\{ U=\binom{u_{1}}{u_{2}}\in \mathbb{V}^1 : -\Delta u_1\in L^2(\Omega),\ \partial_{\bf n}u_1-\Delta_\Gamma u_2\in L^2(\Gamma) \right\}. \label{A_Wentzell}
\end{equation}
By, for example, \cite[see Appendix and in particular Theorem 5.3]{Gal12-1}, the operator $(\Delta_W,D(\Delta_W))$ is self-adjoint and strictly positive operator on $\mathbb{X}^2$, and the resolvent operator $(I+\Delta_W)^{-1}\in \mathcal{L}(\mathbb{X}^2)$ is compact. 
Since $\Gamma$ is of class $\mathcal{C}^2,$ then $D(\Delta_W)=\mathbb{V}^2$. Indeed, the map $L:U\mapsto \Delta_WU,$ as a mapping from $\mathbb{V}^2$ into $\mathbb{X}^2=L^2(\Omega) \times L^2(\Gamma),$ is an isomorphism, and there exists a
positive constant $C_*$, independent of $U=(u,\gamma)^{tr}$, such that, for all $U\in \mathbb{V}^2$, 
\begin{equation}
C_*^{-1}\|U\|_{\mathbb{V}^2}\le \|L(U)\|_{\mathbb{X}^2}\le C_*\|U\|_{\mathbb{V}^2},  \label{regularity-oper}
\end{equation}
(cf. Lemma \ref{t:pre-reg-op}, see also \cite{CFGGGOR09}).

The following basic elliptic estimate is taken from \cite[Lemma 2.2]
{Gal&Grasselli08}.

\begin{lemma}  \label{t:pre-reg-op} 
Consider the linear boundary value problem, 
\begin{equation}
\left\{ \begin{array}{rl} -\Delta u & =p_1 \quad \text{in} \quad \Omega, \\ 
-\Delta_\Gamma u + \partial_{\bf n}u + u & =p_2 \quad \text{on} \quad \Gamma.
\end{array}\right.  \label{pre-reg-BVP}
\end{equation}
If $(p_1,p_2) \in H^s(\Omega)\times H^s(\Gamma)$ for $s\ge0$ and $s+\frac{1}{2}\not\in \mathbb{N}$, then the following estimate holds for some constant $C>0$, 
\begin{equation}
\|u\|_{H^{s+2}(\Omega)}+\|u\|_{H^{s+2}(\Gamma)}\le C\left( \|p_1\|_{H^s(\Omega)}+\|p_2\|_{H^s(\Gamma)}\right).  \label{H2-reg}
\end{equation}
\end{lemma}

We also recall the following basic inequality which gives interior control over some boundary terms (cf. \cite[Lemma A.2]{Gal12-2}).

\begin{lemma}  \label{t:pre-bnd-reg} 
Let $s>1$ and $u\in H^1(\Omega)$. 
Then, for every $\varepsilon>0$, there exists a positive constant $C_\varepsilon\sim\varepsilon^{-1}$ such that, 
\begin{equation}
\|u\|_{L^s(\Gamma)}^s \le \varepsilon\|\nabla u\|_{L^2(\Omega)}^2 + C_\varepsilon(\|u\|_{L^\gamma(\Omega)}^\gamma + 1),  \label{pre-bnd-reg}
\end{equation}
where $\gamma=\max\{s,2(s-1)\}$.
\end{lemma}
We refer the reader to more details to e.g., \cite{CGGM10}, \cite{CFGGGOR09} and \cite{Gal&Warma10} and the references therein. 

Finally, since the operator $\Delta_W$ with domain $D(\Delta_W)$ is positive and self-adjoint on $\mathbb{X}^2$, we may define fractional powers of $\Delta_W$ (see Appendix \ref{s:ap-1}).
Indeed, with $\theta\in[\frac{1}{2},1)$, $\alpha\in(0,1]$ and $\omega\in(0,1]$, we define 
\begin{align}
\Delta_W^\theta \binom{u_1}{u_2} := \begin{pmatrix} \Delta^\theta u_1-u_1 \\ -\partial_{\bf n}^\theta u_1+\Delta_\Gamma u_2-u_2 \end{pmatrix}  \notag
\end{align}
and
\begin{align}
\Delta_W^{\theta,\alpha,\omega} \binom{u_1}{u_2} := \begin{pmatrix} \omega\Delta^\theta u_1-u_1 \\ -\omega\partial_{\bf n}^\theta u_1+\alpha\omega\Delta_\Gamma u_2-u_2 \end{pmatrix}  \notag
\end{align}
with domain
\begin{equation}
D(\Delta_W^{\theta,\alpha,\omega}) := \left\{ U=\binom{u_{1}}{u_{2}}\in \mathbb{V}^1 : -\omega\Delta^\theta u_1\in L^2(\Omega),\ \omega\partial_{\bf n}^\theta u_1-\alpha\omega\Delta_\Gamma u_2\in L^2(\Gamma) \right\}. \label{A_Wentzell2}
\end{equation}
Hence, $\Delta_W^{\theta,1,1}=\Delta_W^\theta$.
The fractional flux $\partial_{\bf n}^\theta$ are defined as follows.
Consider $\partial_{\bf n}u=\nabla u\cdot {\mathbf n}$, and recall $\partial_{\bf n}u\in L^2(\Gamma)$ whenever $u\in H^{3/2}(\Omega)$.
So we can define $\partial_{\bf n}^\theta u=\nabla_W^{\theta/2} u \cdot {\mathbf n}$ when $u\in H^{\frac{1}{2}+\theta}(\Omega)$ guaranteeing the fractional flux $\partial_{\bf n}^\theta u\in L^2(\Gamma).$
(These fractional flux operators are explicitly written in Appendix \ref{s:ap-1}.)
Moving toward the linear operator associated with the model problem \eqref{pde1}-\eqref{ic2}
Let $U=(u_1,u_2)\in \mathbb{V}^1$ and $V=(v_1,v_2)\in \mathbb{X}^2$, and let $\mathcal{X}=(U,V)$.
Motivated by \cite{Carvalho_Cholewa_02}, we define the unbounded linear operator $\mathcal{A}_{\theta,\alpha,\omega}$ written as
\[
\mathcal{A}_{\theta,\alpha,\omega}\mathcal{X} := \begin{pmatrix} 0 & I_{2\times2} \\ \Delta_W & \Delta_W^{\theta,\alpha,\omega} \end{pmatrix}\begin{pmatrix} U \\ V \end{pmatrix} = \begin{pmatrix} V \\ \Delta_W U+\Delta_W^{\theta,\alpha,\omega} V \end{pmatrix} = \begin{pmatrix} V \\ \Delta^{\theta,1,1}_W(\Delta_W^{1-\theta,1,1}U+\Delta_W^{0,\alpha,\omega}V )\end{pmatrix}
\]
with domain
\begin{align}
D(\mathcal{A}_{\theta,\alpha,\omega}) := & \left\{ \mathcal{X}=\binom{U}{V}\in \mathbb{V}^1\times\mathbb{X}^2 : \Delta_W^{1-\theta,1,1}U+\Delta_W^{0,\alpha,\beta}V\in D(\Delta_W^{\theta,1,1}) \right\}.  \notag
\end{align}
By \cite[Theorem 3.1 (a)]{Haraux-Otani-13}, the resolvent $(I_{4\times4}+\mathcal{A}_{\theta,\alpha,\omega})^{-1}\in\mathcal{L}(\mathbb{V}^1\times\mathbb{X}^2)$ is compact.
Hence, we can support the local existence of weak solutions (defined below) with a Galerkin method.

Next we define the nonlinear mapping on $\mathbb{V}^1\times\mathbb{X}^2$ given by 
\begin{align}
F(U):=\binom{0}{-f(u)}, \quad G(U):=\binom{0}{-g(\gamma)}, 
\end{align}
and
\begin{align}
\mathcal{F}(\mathcal{X}) &:= \begin{pmatrix} F(U) \\ G(U) \end{pmatrix} = \begin{pmatrix} 0 \\ -f(u) \\ 0 \\ -g(\gamma) \end{pmatrix} \quad \text{for} \quad U\in\mathbb{V}^1.  \notag
\end{align}
Due to the two embeddings, $H^1(\Omega)\hookrightarrow L^{s_1}(\Omega)$, $s_1\in[1,6]$, and $H^1(\Gamma)\hookrightarrow L^{s_2}(\Omega)$, $s_2\in[1,\infty)$, one can show that when $r_1\in[1,3]$ in \eqref{growth}, then $\mathcal{F}:\mathbb{V}^1\times\mathbb{X}^2\rightarrow\mathbb{V}^1\times\mathbb{X}^2$ is locally Lipschitz (indeed, cf. e.g. \cite[Lemma 2.6]{Graber-Shomberg-16}).
With $r_1\ge1$ arbitrary, this motivates us to set 
\[
\widetilde{\mathbb{V}}^{s,r_1} = \left\{U=(u,\gamma)^{tr}\in \left[ H^s(\Omega)\cap L^{r_1}(\Omega)\right] \times H^s(\Gamma) : tr_D(u)=\gamma\right\}
\]
with the canonical norm whose square is given by
\[
\|U\|^2_{\widetilde{\mathbb{V}}^{s,r_1}}:=\|u\|^2_{H^s(\Omega)}+\|u\|^{r_1}_{L^{r_1}(\Omega)}+\|\gamma\|^2_{H^s(\Gamma)},
\]
and also set $\mathcal{H}_0:=\widetilde{\mathbb{V}}^{1,r_1}\times\mathbb{X}^2$.
The space $\mathcal{H}_0$ is Hilbert with the norm whose square is given by, for $\mathcal{X} = (U,V)\in\mathcal{H}_0$,
\begin{align}
\|\mathcal{X}\|^2_{\mathcal{H}_0} &:= \|U\|^2_{\widetilde{\mathbb{V}}^{1,r_1}}+\|V\|^2_{\mathbb{X}^2} \notag \\
& = \|u\|^2_{H^1(\Omega)}+\|u\|^{r_1}_{L^{r_1}(\Omega)} + \|v\|^2_{L^2(\Omega)} + \|\gamma\|^2_{H^1(\Gamma)} + \|\delta\|^2_{L^2(\Gamma)}   \notag \\ 
& = \left( \|\nabla u\|^2_{L^2(\Omega)} + \|u\|^2_{L^2(\Omega)} \right)+\|u\|^{r_1}_{L^{r_1}(\Omega)} + \|v\|^2_{L^2(\Omega)} + \left( \|\nabla_\Gamma\gamma\|^2_{L^2(\Gamma)} + \|\gamma\|^2_{L^2(\Gamma)} \right) + \|\delta\|^2_{L^2(\Gamma)}.   \notag
\end{align}
The space $\mathcal{H}_0$ is our weak energy phase space.
Moreover, given $\mathcal{X}_0=(U_0,U_1)\in\mathcal{H}_0=\widetilde{\mathbb{V}}^{1,r_1}\times\mathbb{X}^2,$ the abstract formulation of Problem {\textbf P} takes the form
\[
\left\{ \begin{array}{ll}
\displaystyle\frac{d}{dt}\mathcal{X}(t) = \mathcal{A}_{\theta,\alpha,\omega}\mathcal{X}(t)+\mathcal{F}(\mathcal{X}(t)) & t>0, \\ ~ \\
\mathcal{X}(0) = \mathcal{X}_0. \end{array} \right.
\]

We can now introduce the variational formulation of Problem {\textbf P}.

\begin{definition}  \label{d:ws}
Let $\theta\in[\frac{1}{2},1),$ $\alpha\in(0,1]$ and $\omega\in(0,1].$
Let $T>0$ and $\mathcal{X}_0=(U_0,U_1)\in\mathcal{H}_0.$
A function $\mathcal{X}(t)=(U(t),\partial_tU(t))=(u(t),u_{\mid\Gamma}(t),\partial_t u(t),\partial_t u_{\mid\Gamma}(t))$ satisfying  
\begin{align}
U & \in L^\infty(0,T;\mathbb{V}^1),  \label{wk-reg-1} \\
\partial_tU & \in L^\infty(0,T;\mathbb{X}^2),  \label{wk-reg-2} \\
\sqrt{\omega}\partial_tu & \in L^2(0,T;H^\theta(\Omega)),  \label{wk-reg-3} \\
\partial^2_tU & \in L^\infty(0,T;(\mathbb{V}^1)^*),  \label{wk-reg-4} 
\end{align}
for almost all $t\in(0,T]$ is called a {\sc{weak solution}} to Problem {\textbf P} with initial data $\mathcal{X}_0$ if the following identities hold almost everywhere on $[0,T]$, and for all $\Xi=(\Xi_1,\Xi_2)\in \mathbb{V}^1\times\mathbb{V}^1$:
\begin{align}
\frac{d}{dt}\left\langle \mathcal{X}(t),\Xi \right\rangle_{\mathcal{V}^{-1}\times\mathcal{V}^1} = \left\langle \mathcal{A}_{\theta,\alpha,\omega} \mathcal{X}(t),\Xi \right\rangle_{\mathcal{H}_0} + \left\langle \mathcal{F}(\mathcal{X}(t)),\Xi \right\rangle_{\mathcal{H}_0}.  \label{var1}
\end{align}
Also, the initial conditions \eqref{ic1}-\eqref{ic2} hold in the $L^2$-sense; i.e., 
\begin{align}
\left\langle \mathcal{X}(0),\Xi \right\rangle_{\mathcal{H}_0}=\left\langle \mathcal{X}_0,\Xi \right\rangle_{\mathcal{H}_0}, \quad \text{for every}\quad \Xi\in \mathbb{V}^1\times\mathbb{V}^1.  \label{var-ic} 
\end{align}
We say $\mathcal{X}(t)=(U(t),\partial_tU(t))$ is a {\sc{global weak solution}} of Problem {\textbf P} if it is a weak solution on $[0,T]$, for any $T>0.$
\end{definition}

\begin{remark}
Observe that we are solving a more general problem because $\gamma_0$ and $\gamma_1$, from $U_0$ and $U_1$ respectively, may be taken to be initial data {\em independent} of $u$ and $\partial_tu$.
However, if $\partial_tu(t)\in H^{s}(\Omega)$, for all $t>0$ and for some $s>1/2$, then $\gamma_t(t)=\partial_tu_{\mid\Gamma}(t)$.
\end{remark}

\section{Global existence}  \label{s:wp}

\begin{theorem}
Let $\mathcal{X}_0=(U_0,U_1)\in\mathcal{H}_0$ satisfy $\|\mathcal{X}_0\|_{\mathcal{H}_0}\le R$ for some $R>0$.
Then there exists a global weak solution to Problem {\textbf P} satisfying the additional regularity,
\begin{align}
\sqrt{\alpha\omega}\partial_tu & \in L^2(0,T;H^1(\Gamma)), \label{adreg1}
\end{align}
\end{theorem}

\begin{proof}
{\underline{Step 1.}} (An {\em a priori} estimate.) 
In \eqref{var1} take $\Xi=(\partial_tU,\partial_t U)$ to find the differential identity
\begin{align}
\frac{1}{2}\frac{d}{dt} & \left\{ \|\partial_tU\|^2_{\mathbb{X}^2} + \|U\|^2_{\mathbb{V}^1} + 2(F(u),1)_{L^2(\Omega)} + 2(G(u),1)_{L^2(\Gamma)} \right\}  \notag \\ 
& + \omega\|\nabla^\theta \partial_tu\|^2_{L^2(\Omega)} + \|\partial_tu\|^2_{L^2(\Omega)} + \alpha\omega\|\nabla \partial_tu\|^2_{L^2(\Gamma)} + \|\partial_tu\|^2_{L^2(\Gamma)} = 0.  \label{po1}
\end{align}
Using \eqref{growth} and setting $\tilde F'=f$ and $\tilde G'=g$, a simple integration by parts on \eqref{sign} shows, for all $u\in H^1(\Omega)$, and $\gamma\in H^1(\Gamma),$
\begin{align}
(\tilde F(u),1)_{L^2(\Omega)} & \ge (f(u),u)_{L^2(\Omega)}+\frac{M_1}{2}\|u\|^2_{L^2(\Omega)} \label{cons1}
\end{align}
and
\begin{align}
(\tilde G(\gamma),1)_{L^2(\Gamma)} & \ge (g(\gamma),\gamma)_{L^2(\Gamma)}+\frac{M_2}{2}\|\gamma\|^2_{L^2(\Gamma)}. \label{cons2}
\end{align}
To bound the products on the right-hand sides of \eqref{cons1} and \eqref{cons2} from below, we utilize \eqref{balance}.
Following \cite[(2.22)]{Gal12-2}, \cite[(3.34)]{Gal-Shomberg15-2} and \cite[(3.11)]{RBT01}, we estimate the products as 
\begin{align}
& (f(u),u)_{L^2(\Omega)}+(g(u),u)_{L^2(\Gamma)} \notag \\
& = \int_\Omega \left( f(u)u+\frac{|\Gamma|}{|\Omega|}g(u)u \right) dx - \frac{|\Gamma|}{|\Omega|}\int_\Omega \left( g(u)u-\frac{1}{|\Gamma|}\int_\Gamma g(u)u\mathrm{d}\sigma \right) dx,  \label{po3}
\end{align}
whereby we exploit the Poincar\'{e} inequality (\ref{eq:Poincare}) and Young's
inequality to see that, for all $\varepsilon>0$, 
\begin{align}
\frac{|\Gamma|}{|\Omega|}\int_\Omega\left( g(u)u-\frac{1}{
|\Gamma|}\int_\Gamma g(u)u d\sigma \right) dx & \le C_\Omega
\frac{|\Gamma|}{|\Omega|}\int_\Omega|\nabla (g(u)u)|dx \notag \\
& =C_\Omega\frac{|\Gamma|}{|\Omega|}\int_\Omega|\nabla u(g'(u)u+g(u))|dx  \notag \\
& \le \varepsilon \|\nabla u\|_{L^2(\Omega)}^2+\frac{C_\Omega^2|\Gamma|^2}{4\varepsilon |\Omega|^2}\int_\Omega|g'(u)u+g(u)|^2dx. \label{po4}
\end{align}
Then combining \eqref{po3} and \eqref{po4}, and applying assumption \eqref{balance} yields
\begin{equation}
(f(u),u)_{L^2(\Omega)}+(g(u),u)_{L^2(\Gamma)} \ge \|u\|_{L^{r_1}(\Omega)}^{r_1}-\varepsilon \|\nabla u\|_{L^2(\Omega)}^2-C_\delta,  \label{po5}
\end{equation}
for some positive constants $\delta$ and $C_\delta$ that are independent
of $t$ and $\varepsilon$.
Hence, together \eqref{cons1} and \eqref{cons2} become
\begin{align}
(F(u),1)_{L^2(\Omega)}+(G(u),1)_{L^2(\Gamma)} & \ge \|u\|_{L^{r_1}(\Omega)}^{r_1}+\frac{M_1}{2}\|u\|^2_{L^2(\Omega)}+\frac{M_2}{2}\|u\|^2_{L^2(\Gamma)}-\varepsilon \|\nabla u\|_{L^2(\Omega)}^2-C_\delta. \label{cons3}
\end{align}
Moreover, \eqref{cons3} provides a lower-bound to the functional 
\begin{align}
E(t) := \|\partial_tU(t)\|^2_{\mathbb{X}^2} + \|U(t)\|^2_{\mathbb{V}^1} + 2(F(u(t)),1)_{L^2(\Omega)} + 2(G(u(t)),1)_{L^2(\Gamma)}.  \notag
\end{align}
Integrating the identity \eqref{po1} over $(0,t)$, yields
\begin{align}
E(t) + 2\int_0^t \left( \omega\|\nabla^\theta \partial_tu(\tau)\|^2_{L^2(\Omega)} + \alpha\omega\|\nabla \partial_tu(\tau)\|^2_{L^2(\Gamma)} + \|\partial_tU(\tau)\|^2_{\mathbb{X}^2} \right)d\tau = E(0).  \label{po1.5}
\end{align}
We can find an upper-bound on $E(0)$ with \eqref{growth}.
Evidently
\begin{align}
& 2(F(u(0)),1)_{L^2(\Omega)} + 2(G(u(0)),1)_{L^2(\Gamma)} \notag \\
& \le \ell_1(\|u(0)\|_{L^1(\Omega)}+\|u(0)\|^{r_1}_{L^{r_1}(\Omega)})+\ell_2(\|u(0)\|_{L^1(\Gamma)}+\|u(0)\|^{r_2}_{L^{r_2}(\Gamma)}).  \label{po1.6}
\end{align}
Hence, \eqref{po1.6} and the embedding $\mathbb{V}^1\hookrightarrow\mathbb{X}^2$ show
\begin{align}
E(0) & \le \|\partial_tu(0)\|^2_{L^2(\Omega)} + \|\nabla u(0)\|^2_{L^2(\Omega)} + \|u(0)\|^2_{L^2(\Omega)} + \|\partial_tu(0)\|^2_{L^2(\Gamma)} + \|\nabla u(0)\|^2_{L^2(\Gamma)} + \|u(0)\|^2_{L^2(\Gamma)}  \notag \\ 
& + \ell_1(\|u(0)\|_{L^1(\Omega)}+\|u(0)\|^{r_1}_{L^{r_1}(\Omega)})+\ell_2(\|u(0)\|_{L^1(\Gamma)}+\|u(0)\|^{r_2}_{L^{r_2}(\Gamma)}) \notag \\
& \le \|\partial_tU(0)\|^2_{\mathbb{X}^2} + \|U(0)\|^2_{\mathbb{V}^1} + C\left( \|U(0)\|_{\mathbb{V}^1} + \|u(0)\|^{r_1}_{L^{r_1}(\Omega)} + \|u(0)\|^{r_2}_{L^{r_2}(\Gamma)} \right).  \label{po1.7}
\end{align}
Thus \eqref{po1.5} and \eqref{po1.7} yield, for all $t\ge0,$
\begin{align}
& \|\partial_tU(t)\|^2_{\mathbb{X}^2} + \|U(t)\|^2_{\mathbb{V}^1} + 2(F(u(t)),1)_{L^2(\Omega)} + 2(G(u(t)),1)_{L^2(\Gamma)} \notag \\
& + 2\int_0^t \left( \omega\|\nabla^\theta \partial_tu(\tau)\|^2_{L^2(\Omega)} + \alpha\omega\|\nabla \partial_tu(\tau)\|^2_{L^2(\Gamma)} + \|\partial_tU(\tau)\|^2_{\mathbb{X}^2} \right)d\tau \notag \\
& \le \|\partial_tU(0)\|^2_{\mathbb{X}^2} + \|U(0)\|^2_{\mathbb{V}^1} + C\left( \|U(0)\|_{\mathbb{V}^1} + \|u(0)\|^{r_1}_{L^{r_1}(\Omega)} + \|u(0)\|^{r_2}_{L^{r_2}(\Gamma)} \right) \notag \\
& \le \|\partial_tU(0)\|^2_{\mathbb{X}^2} + \|U(0)\|^2_{\mathbb{V}^1} + C\left( \|U(0)\|_{\mathbb{V}^1} + \|u(0)\|^{r_1}_{L^{r_1}(\Omega)} + 1 \right), \label{po1.8}
\end{align}
where the last inequality follows from Lemma \ref{t:pre-bnd-reg}.

Now we see that, for any $T>0$, there hold
\begin{align}
U & \in L^\infty(0,T;\mathbb{V}^1),  \label{po6} \\
\partial_tU & \in L^\infty(0,T;\mathbb{X}^2),  \label{po7} \\
\sqrt{\omega} \partial_tu & \in L^2(0,T;H^\theta(\Omega)), \label{po7.1} \\
\sqrt{\alpha\omega} \partial_tu & \in L^2(0,T;H^1(\Gamma)), \label{po7.3} \\
F(u) & \in L^\infty(0,T;L^1(\Omega)), \label{po8} \\
G(u) & \in L^\infty(0,T;L^1(\Gamma)). \label{po9} 
\end{align}
We have found $\mathcal{X}\in L^\infty(0,T;\mathcal{H}_0)$.
Moreover, since $U\in L^\infty(0,T;\mathbb{V}^1)$, we have $\Delta_WU \in L^\infty(0,T;(\mathbb{V}^1)^*)$ and as $\sqrt{\alpha\omega}\partial_tU\in L^2(0,T;\mathbb{V}^1)$, we also have $\Delta_W^{\theta,\alpha,\omega}\partial_tU \in L^2(0,T;(\mathbb{V}^1)^*)$.
Therefore, after comparing terms in the first equation of \eqref{po1}, we see that 
\begin{equation}
\partial_t^2U\in L^2(0,T;(\mathbb{V}^1)^*). \label{po10}
\end{equation}
Hence, this justifies our choice of test function in \eqref{po1}. 
With \eqref{po7.3}, we also find \eqref{adreg1} as claimed.
This concludes Step 1.

{\underline{Step 2.}} (A Galerkin basis.)
According to Section \ref{s:prelim}, for each $\theta\in[\frac{1}{2},1),$ the operator $\mathcal{A}_{\theta,\alpha,\omega}$ admits a system of eigenfunctions $\Psi_i^{\theta,\alpha,\omega}=(\psi^{\theta,\alpha,\omega},\phi^{\theta,\alpha,\omega},\psi_{\mid\Gamma}^{\theta,\alpha,\omega},\phi_{\mid\Gamma}^{\theta,\alpha,\omega})$ satisfying $\{\Psi_i^{\theta,\alpha,\omega}\}_{i=1}^\infty \subset D(\mathcal{A}_{\theta,\alpha,\omega})\cap (C^2({\overline{\Omega}}) \times C^2(\Gamma) \times C^2({\overline{\Omega}}) \times C^2(\Gamma))$ and 
\[
\mathcal{A}_{\theta,\alpha,\omega} \Psi_i^{\theta,\alpha,\omega}=\Lambda_i\Psi_i^{\theta,\alpha,\omega}, \quad i=1,2,\dots,
\]
where the eigenvalues $\Lambda_i=\Lambda_i^{\theta,\alpha,\omega}\in(0,+\infty)$ may be put into increasing order and counted according to their multiplicity to form a diverging sequence.
This means the pair $(\Lambda_i,\Psi_i)$, $\Psi_i=\Psi_i^{\theta,\alpha,\omega}$ is a classical solution of the elliptic problem
\[
\left\{ \begin{array}{ll} 
-\Delta\psi_i + \psi_i + \omega(-\Delta)^\theta \phi_i + \phi_i  = \Lambda_i\psi_i & \text{in}\ \Omega \\ 
-\alpha\omega\Delta_\Gamma \phi_{i\mid\Gamma} + \phi_{i\mid\Gamma} - \Delta_\Gamma\psi_{i\mid\Gamma} + \psi_{i\mid\Gamma} = \Lambda_i\psi_{i\mid\Gamma} & \text{on}\ \Gamma. \end{array} \right.
\]
Also due to standard spectral theory, these eigenfunctions form an orthogonal basis in $\mathcal{H}_0$ that is orthonormal in $L^2(\Omega)\times L^2(\Omega)\times L^2(\Gamma)\times L^2(\Gamma)$. 

Let $T>0$ be fixed. 
For $n\in \mathbb{N}$, set the spaces 
\begin{equation*}
\mathbb{H}_n := \mathrm{span}\left\{ \Psi_1^{\theta,\alpha,\omega},\dots,\Psi_n^{\theta,\alpha,\omega}\right\} \subset \mathcal{H}_0 \quad \text{and} \quad \mathbb{H}_\infty := \bigcup_{n=1}^\infty \mathbb{H}_n.
\end{equation*}
Obviously, $\mathbb{H}_\infty$ is a dense subspace of $\mathcal{H}_0$. 
For each $n\in \mathbb{N}$, let $\mathbb{P}_n:\mathcal{H}_0\rightarrow \mathbb{H}_n$ denote the orthogonal projection of $\mathcal{H}_0$ onto $\mathbb{H}_n$. 
Thus, we seek functions of the form 
\begin{equation}
\mathcal{X}^{(n)}(t)=\sum_{i=1}^n A_i(t)\Psi_i^{\theta,\alpha,\omega}  \label{po11}
\end{equation}
that will satisfy the associated discretized Problem \textbf{P}$_n$
described below. 
The functions $A_i$ are assumed to be (at least) $C^2((0,T))$ for $i=1,\dots,n$. 
Precisely, 
\begin{equation}
u^{(n)}(t)=\sum_{i=1}^n A_i(t)\psi_i^{\theta,\alpha,\omega}, \quad \partial_t u^{(n)}(t)=\sum_{i=1}^n A_i'(t)\psi_i^{\theta,\alpha,\omega}, \label{po12-1}
\end{equation}
and
\begin{equation}
u^{(n)}_{\mid\Gamma}(t)=\sum_{i=1}^n A_i(t)\phi_{i\mid\Gamma}^{\theta,\alpha,\omega}, \quad \partial_t u^{(n)}_{\mid\Gamma}(t)=\sum_{i=1}^n A_i'(t)\phi_{i\mid\Gamma}^{\theta,\alpha,\omega}. \label{po12-2}
\end{equation}
Using semigroup properties of $\mathcal{A}_{\theta,\alpha,\omega}$, the domain $D(\mathcal{A}^{\theta,\alpha,\omega})$ is dense in $\mathcal{H}_0$.
So to approximate the given initial data $\mathcal{X}_0\in \mathcal{H}_0$, we may take $\mathcal{X}_0^{(n)}\in D(\mathcal{A}^{\theta,\alpha,\omega})$ such that $\mathcal{X}_0^{(n)}\rightarrow \mathcal{X}_0$ in $\mathcal{H}_0$.

For $T>0$ and for each integer $n\ge1$, the weak formulation of the
approximate Problem \textbf{P}$_n$ is: to find $\mathcal{X}^{(n)}$ given by \eqref{po11} such that, for all ${\overline{\mathcal{X}}}=(\overline{U},\overline{V})\in \mathbb{H}_n$, the equation
\begin{equation}
\left\langle \partial_t\mathcal{X}^{(n)},{\overline{\mathcal{X}}}\right\rangle_{\mathcal{H}_0} + \left\langle \mathcal{A}_{\theta,\alpha,\omega}\mathcal{X}^{(n)},{\overline{\mathcal{X}}} \right\rangle_{\mathcal{H}_0}+\left\langle \mathbb{P}_n\mathcal{F}\left( \mathcal{X}^{(n)}\right),{\overline{\mathcal{X}}}\right\rangle_{\mathcal{H}_0}=0 \label{po13}
\end{equation}
holds for almost all $t\in(0,T)$, subject to the initial conditions 
\begin{equation}
\left\langle \mathcal{X}^{(n)}(0),{\overline{\mathcal{X}}} \right\rangle _{\mathcal{H}_0}=\left\langle \mathcal{X}_0^{(n)},{\overline{\mathcal{X}}}\right\rangle_{\mathcal{H}_0}.  \label{po14}
\end{equation}

To show the existence of at least one solution to \eqref{po13}-\eqref{po14}, we now suppose that $n$ is fixed and we take ${\overline{\mathcal{X}}}=\mathcal{X}^{(k)}$ for some $1\le k\le n$. 
Then substituting the discretized functions \eqref{po12-1}-\eqref{po12-2} into \eqref{po13}-\eqref{po14}, we find a system of ordinary differential equations in the unknowns $A_k=A_k(t)$ on $\mathcal{X}^{(n)}$. 
Also, we recall that
\begin{equation*}
\left\langle \mathbb{P}_n\mathcal{F}\left( \mathcal{X}^{(n)}\right),\mathcal{X}^{(k)} \right\rangle_{\mathcal{H}_0}=\left\langle \mathcal{F}\left( \mathcal{X}^{(n)}\right),\mathbb{P}_n\mathcal{X}^{(k)} \right\rangle_{\mathcal{H}_0}=\left\langle \mathcal{F}\left( \mathcal{X}^{(n)}\right),\mathcal{X}^{(k)} \right\rangle_{\mathcal{H}_0}.
\end{equation*}
Since $f\in C(\mathbb{R})$ and $g\in C^1(\mathbb{R})$, we may apply Cauchy's theorem for ODEs to find that there is $T_n\in (0,T)$ such that $A_k\in C^2((0,T_n))$, for $1\le k\le n$, and \eqref{po13} holds in the classical sense for all $t\in [0,T_n]$. 
This shows the existence of at least one local solution to the approximate Problem \textbf{P}$_n$ and ends Step 2.

\underline{Step 3}. (Boundedness and continuation of approximate maximal
solutions.) 
We begin by noticing that the {\em a priori} estimate \eqref{po1.8} holds for any approximate solution $\mathcal{X}^{(n)}$ of Problem \textbf{P}$_n$ on the interval $[0,T_n)$, where $T_n<T$. 
Thanks to the boundedness of the projector $\mathbb{P}_n$, we infer
\begin{align}
& \|\partial_tU^{(n)}(t)\|^2_{\mathbb{X}^2} + \|U^{(n)}(t)\|^2_{\mathbb{V}^1} + 2(F(u^{(n)}(t)),1)_{L^2(\Omega)} + 2(G(u^{(n)}(t)),1)_{L^2(\Gamma)} \notag \\
& + 2\int_0^t \left( \omega\|\nabla^\theta \partial_tu^{(n)}(\tau)\|^2_{L^2(\Omega)} + \alpha\omega\|\nabla \partial_tu^{(n)}(\tau)\|^2_{L^2(\Gamma)} + \|\partial_tU^{(n)}(\tau)\|^2_{\mathbb{X}^2} \right)d\tau \notag \\
& \le \|\partial_tU(0)\|^2_{\mathbb{X}^2} + \|U(0)\|^2_{\mathbb{V}^1} + C\left( \|U(0)\|_{\mathbb{V}^1} + \|u(0)\|^{r_1}_{L^{r_1}(\Omega)} + \|u(0)\|^{r_2}_{L^{r_2}(\Gamma)} \right).  \label{po15}
\end{align}
Since the right-hand side of \eqref{po15} is independent of $n$ and $t$, every
approximate solution may be extended to the whole interval $[0,T]$, and
because $T>0$ is arbitrary, any approximate solution is a global one. 
From above in Step 1, we also obtain the uniform bounds \eqref{po6}-\eqref{po10} for each approximate solution $\mathcal{X}^{(n)}$.
Thus,
\begin{align}
U^{(n)} & \in L^\infty(0,T;\mathbb{V}^1),  \label{po16} \\
\partial_tU^{(n)} & \in L^\infty(0,T;\mathbb{X}^2),  \label{po17} \\
\sqrt{\omega} \partial_tu^{(n)} & \in L^2(0,T;H^\theta(\Omega)), \label{po17.1} \\
\sqrt{\alpha\omega} \partial_tu^{(n)} & \in L^2(0,T;H^1(\Gamma)), \label{po17.3} \\
F(u^{(n)}) & \in L^\infty(0,T;L^1(\Omega)), \label{po18} \\
G(u^{(n)}) & \in L^\infty(0,T;L^1(\Gamma)). \label{po19} 
\end{align}
This concludes Step 3.

\underline{Step 4}. (Convergence of approximate solutions.) 
We begin this step by applying Alaoglu's theorem (cf. e.g. \cite[Theorem 6.64]{Renardy&Rogers04}) to the uniform bounds \eqref{po16}-\eqref{po19} to find that there is a subsequence of $\mathcal{X}^{(n)}$, generally not relabelled, and a function $\mathcal{X}=(u,\partial_tu,u_{\mid\Gamma},\partial_tu_{\mid\Gamma})$, obeying \eqref{po6}-\eqref{po10}, such that as $n\rightarrow \infty$, 
\begin{align}
U^{(n)}\rightharpoonup U & \quad\text{weakly$-*$ in}\quad L^\infty(0,T;\mathbb{V}^1), \label{po20} \\ 
\partial_tU^{(n)}\rightharpoonup \partial_tU & \quad\text{weakly$-*$ in}\quad L^\infty(0,T;\mathbb{X}^2), \label{po21} \\ 
\sqrt{\omega}\partial_tu^{(n)}\rightharpoonup u & \quad\text{weakly in}\quad L^2(0,T;H^\theta(\Omega)), \label{po22} \\ 
\sqrt{\alpha\omega}\partial_tu^{(n)}\rightharpoonup u & \quad\text{weakly in}\quad L^2(0,T;H^1(\Gamma)), \label{po23} \\
\partial_tU^{(n)}\rightharpoonup \partial_tU & \quad\text{weakly in}\quad L^2(0,T;(\mathbb{V}^1)^*). \label{po24}
\end{align}
Using the above convergences \eqref{po20} and \eqref{po21}, as well as the fact that the injection $\mathbb{V}^1\hookrightarrow\mathbb{X}^2$ is compact, we draw upon the conclusion of the Aubin-Lions Lemma (cf. Lemma \ref{t:Lions}) to deduce the following embedding is compact
\begin{align}
\mathcal{W}:=\{ U\in L^2(0,T;\mathbb{V}^1) : \partial_tU\in L^2(0,T;\mathbb{X}^2)\} \hookrightarrow L^2(0,T;\mathbb{X}^2) \label{compact-u}
\end{align}
(see, e.g., \cite{Temam88}).
Thus,
\begin{eqnarray}
U^{(n)} \rightarrow U & \text{strongly in} & L^2(0,T;\mathbb{X}^2), \label{st-con-u}
\end{eqnarray}
and deduce that $U^{(n)}$ converges to $U$, almost everywhere in $\Omega\times (0,T)$. 
The last strong convergence property is enough to pass to the limit in the
nonlinear terms since $f,g\in C^1(\mathbb{R})$ (see, e.g., \cite{Gal12-2,Gal&Warma10}). 
Indeed, on account of standard arguments (cf. also \cite{CGGM10}) we have
\begin{eqnarray}
\mathbb{P}_n\mathcal{F}(\mathcal{X}^{(n)}) \rightharpoonup \mathcal{F}(\mathcal{X}) & \text{weakly in} & L^2(0,T;\mathcal{H}_0).  \label{po25}  
\end{eqnarray}
At this point the convergence properties \eqref{po20}-\eqref{po25} are sufficient to pass to the limit as $n\rightarrow\infty$ in equation \eqref{po13}.
Additionally, we recover \eqref{var1} using standard density arguments. 
The proof of the theorem is finished.
\end{proof}

\noindent {\bf{Concerning uniqueness.}}
A proof of the following conjecture is needed to show that the weak solutions to Problem {\textbf{P}} constructed above depend continuously on initial data, and hence, are unique. 

\begin{conjecture} 
Let $T>0$, $R>0$ and $\mathcal{X}_{01}=(U_{01},U_{11}),\mathcal{X}_{02}=(U_{02},U_{12})\in\mathcal{H}_0$ be such that $\|\mathcal{X}_{01}\|_{\mathcal{H}_0}\le R$ and $\|\mathcal{X}_{02}\|_{\mathcal{H}_0}\le R$.
Any two weak solutions, $\mathcal{X}^1(t)$ and $\mathcal{X}^2(t)$, to Problem {\textbf{P}} on $[0,T]$ corresponding to the initial data $\mathcal{X}_{01}$ and $\mathcal{X}_{02}$, respectively, satisfy for all $t\in[0,T]$,
\begin{equation}  \label{contdep}
\left\|\mathcal{X}^1(t)-\mathcal{X}^2(t)\right\|_{\mathcal{H}_0} \le e^{Q(R)t} \left\|\mathcal{X}_{01} - \mathcal{X}_{02}\right\|_{\mathcal{H}_0}.
\end{equation}
\end{conjecture}

In order to prove the conjecture, typically one needs to control products of the form
\[
(f(u^1)-f(u^2),\partial_t\bar u)_{L^2(\Omega)} \quad \text{and} \quad (g(u^1)-g(u^2),\partial_t\bar u)_{L^2(\Gamma)}
\]
where $u^1$ and $u^2$ are two weak solutions corresponding to (possibly the same) data $\mathcal{X}_{01}=(U_{01},U_{11})=(u_{01},\gamma_{01},u_{11},\gamma_{11})$ and $\mathcal{X}_{02}=(U_{02},U_{12})=(u_{02},\gamma_{02},u_{12},\gamma_{12})$.
A suitable control on $\|f(u^1)-f(u^2)\|_{L^q(\Omega)}$, for example, is readily available when we assume \eqref{growth} with $r_1\in[1,3]$ (cf. \cite[Lemma 2.6]{Graber-Shomberg-16})), but this is no longer valid when we assume $r_1\ge1$ is arbitrary.
In the later case it would be interesting to investigate whether a {\em generalized semiflow} in the sense of \cite{Ball00,Ball04} exists.
Under certain conditions, such generalized semiflows admit global attractors which have similar properties to their well-posed counterparts (cf. \cite{Hale88}).

\appendix
\section{}  \label{s:ap-1}

As introduced in Section \ref{s:prelim}, the Wentzell-Laplacian $\Delta_W$ on $\mathbb{X}^2$ with domain 
\begin{align}
D(\Delta_W):=\{U=(u,\gamma)^{tr}\in\mathbb{V}^1:-\Delta u\in L^2(\Omega), \partial_{\bf n}u=-\gamma+\Delta_\Gamma\gamma \in L^2(\Gamma), \gamma=tr_D(u)\}.  \notag
\end{align}
is positive, self-adjoint and has compact resolvent \cite{AreMetPalRom-2003}.
From \cite[Theorem A.37 (Spectral Theorem) and (A.28)]{Milani&Koksch05}, we know that for each $\theta\in[\frac{1}{2},1)$,
\[
D(\Delta_W^\theta)=\left\{U=(u,\gamma)^{tr}\in D(\Delta_W):\sum_{j=1}^\infty \Lambda_j^{2\theta}|(U,W_j)|^2<\infty \right\}\quad \text{where} \quad \Delta_W W_j=\Lambda_jW_j,
\]
and the sequence $(\Lambda_j)_{j=1}^\infty$ contains real, strictly positive eigenvalues, each having finite multiplicity, which can be ordered into a nondecreasing sequence in which
\[
\lim_{j\rightarrow\infty} \Lambda_j = +\infty.
\]
We mention some results \cite[Theorem 5.2 (c)]{Gal12-1} concerning the regularity of the eigenfunctions $W_j$.
If $\Gamma$ is Lipschitz, then every eigenfunction $W_j\in \mathcal{V}^1$, and in fact $W_j\in C({\overline{\Omega}})\cap C^\infty(\Omega)$, for every $j$.
If $\Gamma$ is of class $\mathcal{C}^2$, then every eigenfunction $W_j\in \mathcal{V}^1\cap C^2({\overline{\Omega}})$ for every $j$.

Here we remind the reader how we define the fractional powers of the Wentzell-Laplacian with a Fourier series.
Thus,
\[
\Delta_W^\theta U=\sum_{j=1}^\infty \Lambda_j^{2\theta}(U,W_j)W_j,
\]
and we can rely on (cf. \cite[(2.6)]{dovgar}) to define the fractional flux, where,
\[
\Delta_W^{\theta/2}U= \nabla_W^\theta U=\sum_{i=1}^N\frac{\partial^\theta U}{\partial x^\theta_i}{\bf e}_i,
\]
and
\[
\frac{d^\theta U}{dx^\theta}=\frac{1}{\Gamma(1-\theta)}\frac{d}{dx}\int_{-\infty}^x(x-y)^{-\theta}U(y)dy.
\]

The following result is the classical Aubin-Lions Lemma, reported here for the reader's convince (cf. \cite{Lions69}, and, e.g. \cite[Lemma 5.51]{Tanabe79} or \cite[Theorem 3.1.1]{Zheng04}). 

\begin{lemma}  \label{t:Lions}
Let $X,Y,Z$ be Banach spaces where $Z\hookleftarrow Y\hookleftarrow X$ with continuous injections, the second being compact.
Then the following embeddings are compact:
\[
W:=\{ \chi\in L^2(0,T;X), \ \partial_t\chi\in L^2(0,T;Z) \} \hookrightarrow L^2(0,T;Y).
\]
\end{lemma}

\section*{Acknowledgments}

The author is grateful to the anonymous referees for their careful reading of the manuscript and for their helpful comments and suggestions.


\providecommand{\bysame}{\leavevmode\hbox to3em{\hrulefill}\thinspace}
\providecommand{\MR}{\relax\ifhmode\unskip\space\fi MR }
\providecommand{\MRhref}[2]{%
  \href{http://www.ams.org/mathscinet-getitem?mr=#1}{#2}
}
\providecommand{\href}[2]{#2}

\end{document}